\documentclass{amsart}
\usepackage{euscript}           
\usepackage{amsmath,amsthm}     
\usepackage{amssymb}            
\usepackage[usenames,dvipsnames]{color}
\usepackage{graphicx}
\usepackage[frame,matrix,curve, arrow]{xy}      

\xymatrixcolsep{1.9pc}                          
\xymatrixrowsep{1.9pc}
\newdir{ >}{{}*!/-5pt/\dir{>}}                  
\newdir{ |}{{}*!/-5pt/\dir{|}}                  

\newfont{\german}       {eufm10 at 12pt}

\numberwithin{equation}{section}

\renewenvironment{proof}[1][]{\begin{trivlist} \item[\hskip\labelsep 
\textit{Proof#1.}]}{\hfill\qedsymbol \end{trivlist}}

\newcommand{\noqed}{\renewcommand{\qedsymbol}{}}
\newcommand{\auxqed}[1][1.5]{\noqed\vskip -#1\baselineskip\hfill$\square$}
\newcommand{\disqed}[1][1]{%
\vskip -#1\baselineskip\hfill$\square$\vskip #1\baselineskip}

\newcommand{\comment}[1]{}

\newcommand{\ra}{\rightarrow}                   
\newcommand{\lra}{\longrightarrow}              

\DeclareMathOperator{\Hom}{Hom} 

\newcommand {\CC}{{\mathbb C}}

\newcommand {\F}{{\mathbb F}}

\newcommand {\Z}{{\mathbb Z}}

\newcommand {\mk}{{\mathbb F}}

\newcommand{\can}{\tau_{T_2}}
\renewcommand{\can}{can}
\newcommand{\cat}{\mathcal}    
\newcommand{\id}{{\mathrm{id}}}
\newcommand{\tmf}{\text{\it tmf\/}}

\newtheorem{thm}{Theorem}[section]
\newtheorem{prop}[thm]{Proposition}

\newtheorem{lemma}[thm]{Lemma}

\theoremstyle{definition}  
\newtheorem{definition}[thm]{Definition}

\newtheorem{convention}[thm]{Convention}

\begin{document}

\title{Towards a splitting of the $K(2)$-local string bordism spectrum}
\author{Gerd Laures and Bj\"orn Schuster}

\address{ Fakult\"at f\"ur Mathematik,  Ruhr-Universit\"at Bochum, NA1/66, D-44780 Bochum, Germany}



\date{\today}

\begin{abstract}
We show that $K(2)$-locally, the smash product of the string bordism spectrum and the spectrum 
$T_2$ splits into copies of Morava $E$-theories. Here, $T_2$ is related to the Thom spectrum 
of the canonical bundle over $\Omega SU(4)$.
\end{abstract}

\maketitle
 
\section{Introduction and statement of results}
In the late 80s \cite{MR970288}, Witten gave an interpretation of the level one elliptic genus as 
the $S^1$-equivariant index of the Dirac operator on the free loop space $LM$ of a manifold. 
He needed the loop space of the manifold to carry a spin structure. This is guaranteed if the 
classifying map of its stable tangent bundle lifts to $BString$, the $7$-connected cover of $BO$. 
This classifying space and its associated Thom spectrum $MString$ have been extensively 
studied by homotopy theorists already in the 70s. 
At that time, Anderson, Brown and Peterson had just succeeded in calculating the spin bordism groups via a now famous splitting of the Thom spectrum $MSpin$ \cite{MR0219077}.
The string bordism groups however 
have not been calculated yet. It is known, that they have torsion at the primes 2 and 3, 
and the hope is to get new insights using  the Witten genus and the arithmetic of modular forms.

At the prime~3,  Hovey and Ravenel \cite[Corollary 2.2]{MR1297530}  have shown that the 
product $DA(0)\wedge MString$ is a wedge of suspensions of the  Brown-Peterson 
spectrum $BP$. Here, $DA(0)$ is the 8-skeleton of $BP$; it is a 3-cell complex which is free 
over $P(0)$, the sub-Hopf algebra of the the mod 3 Steenrod algebra ${\cat A}$
generated by~$P^1$.
 
At the prime~2, they have shown that $DA(1)\wedge MU\langle 6\rangle$ splits into a wedge 
of suspensions of $BP$s, where $MU\langle 6\rangle$ is the Thom spectrum of 
the 5-connected cover $BU\langle 6\rangle$ of $BU$ and $DA(1)$ is an eight cell 
complex whose cohomology is free over $P(1)$, the double of $A(1)$.

This paper is concerned with a similar decomposition of $MString$ at the prime~2. 
Starting point is the spectrum of topological modular forms $\tmf$ which was introduced 
by Goerss, Hopkins and Miller, see \cite{MR2648680} or \cite{MR3223024}. 
Its mod 2 cohomology is the quotient ${\cat A}/ \!/ {\cat A}(2)$ which is known to be a 
direct summand in the cohomology of $MString$ (cf. \cite{MR550517}).  
There even is a ring map from $MString$ to $\tmf$ which induces the Witten genus 
mentioned earlier (cf. \cite{AHR10}). Moreover,  the 2-local equivalence 
$$DA(1)\wedge \tmf  \simeq BP \langle 2\rangle$$ 
(cf.~\cite{MR3515195})  looks encouraging when looking for new splittings 
of products of known spectra with $MString$ into complex orientable spectra at the prime 2.  

When dealing with objects from derived algebraic geometry, it is convenient to work with 
ring spectra rather than spectra. Ravenel introduced an important filtration 
$\{ X(n)\}_{n\in {\mathbb N}}$ of $MU$ by ring spectra which arise from the thomification 
of the filtration 
$$\Omega SU(1)\subset \Omega SU(2) \cdots \subset \Omega SU =BU\,.$$
A multiplicative map from $X(n)$ to a complex orientable spectrum $E$ corresponds to 
a complex orientation of $E$ up to degree $n$ . 
Locally at a prime $p$, $X(n)$ is equivalent to a wedge of suspensions of spectra $T(m)$ 
with $p^m\leq n <p^{m+1}$. We refer to \cite[Section 6.5] {MR860042} for a full report on these spectra. They can be regarded as  $p$-typical versions of the spectra $X(n)$:  they filter the Brown-Peterson spectrum $BP$ in the same way as the spectra $X(n)$ filter the complex bordism spectrum $MU$.  In addition, they satisfy
$$ \pi_* E_*\wedge T(m)\cong E_*[t_1,t_2,\ldots ,t_m]$$
for all complex orientable $E$. A
multiplicative map from $T(m)$ to $E$ corresponds to a $p$-typical orientation up to degree $m$. 
The homotopy groups of $T(m)$ coincide with the homotopy groups of $BP$ up to dimension $2(p^m-2)$.  

For $p=2$, the homotopy groups of $T(2)$ coincide with those of $BP$ up to 
dimension 12. There is a  map from the twelve dimensional even complex $DA(1)$ to $BP$ 
which induces an isomorphism in $\pi_0$ (see for instance \cite{MR3515195}). It is unique in 
mod 2 homology and  lifts to a map
$$ DA(1)\lra T(2)$$
by what was just said. Hence, when looking for splitting results we may replace the finite spectrum $DA(1)$ with the reasonably small spectra $T(2)$ or $X(4)$. Along these lines, we 
mention that
$$ \pi_* X(4)\wedge \tmf\cong \Z[a_1,a_2,a_3,a_4,a_6]$$
carries the  Weierstrass Hopf algebroid  (cf. \cite[Chapter 9]{MR3223024}).

In order to state the main theorem, we work in the $K(2)$-local category and omit the localization 
functor from the notation. In this category  $BP$ splits into a sum of Johnson-Wilson spectra $E(2)$ (see \cite{MR1722151}). We shall write $T_2$ for the even periodic version of $T(2)$ 
by which we mean the following: the class $v_2$ of degree 6 is a unit in the local $T(2)$ 
to which we may associate a root by setting
\begin{eqnarray}\label{T2}
 T_2&=&T(2)[u^\pm]/(u^3-v_2).
 \end{eqnarray}
Note that $T(2)$ is related to $E(2)$ in the same way as $T_2$ to the Morava $E$-theory 
spectrum $E_2$ so that the notation fits well.


\begin{thm}\label{main}
$K(2)$-locally at the prime 2, there is a splitting of 
$ T_2 \wedge MString$ into a wedge of copies of $E_2$.
\end{thm}
The proof of the theorem is in spirit Thom's and Wall's original proof for the splitting of $MO$ and $MSO$ (or even the refinements by Anderson-Brown-Peterson for $MSpin$).
It
uses a generalized Milnor-Moore argument which 
we believe to be of independent interest. One version of the original Milnor-Moore theorem 
states that a  graded connected Hopf algebra is free as a module over any of its 
sub-Hopf algebras.  There are dual  versions for surjective coalgebra maps and 
Hopf algebroids \cite[A1.1.17]{MR860042},
but they all use  connectivity and the grading. We will show a version which only uses the 
coradical filtration of pointed coalgebras. The  assumptions needed to make it work
are automatically satisfied  if the comodules are graded and connected. This generalization applies in particular to
the case of the Hopf algebra $\Sigma=\pi_0(K(n)\wedge E_n)$. We mention that that for $n=1$ it can be used to give a new elementary proof of the Anderson-Brown-Peterson splitting of $MSpin$ into sums of $KO$s in the $K(1)$-local setting.\par
In the $K(2)$-local setting, we prove a key lemma which allows us to switch between the complex and the Witten orientation for $MU\langle 6 \rangle$ by a comodule algebra automorphism. It relies on results of Ando-Hopkins-Strickland  \cite{MR1869850} and  supplies the requirements of the generalized Milnor-Moore theorem. Finally, we show that in the $K(2)$-local categeory a spectrum already splits into sums of copies of $E_2$ if its $K(2)$-homology splits as a comodule over $\Sigma$.

\subsubsection*{Acknowledgements.}
We would like to thank the referee for a careful revision.

\section{A Milnor-Moore type theorem}
Let $C$ be a coalgebra over a field $\mk$. Assume that $C$ is pointed, that is, all simple 
subcoalgebras are 1-dimensional. (Recall that a coalgebra is simple if  any proper subcoalgebra is trivial.) The coradical $R$ of a pointed coalgebra is generated by the set $G$ of grouplike 
elements (cf. \cite[p.182]{MR0252485}) and
$$R=\mk[G]$$ 
is a subcoalgebra of $C$.
The iterated coproduct defines an increasing filtration
$$F_k =\ker(C \lra C^{\otimes (k+1)}\lra (C/R)^{\otimes (k+1)})$$
which is  called  the {\em coradical  filtration}. 
By definition,  $F_0=R$. 
We  call a filtration {\em exhaustive} if every element of $C$ lies in some $F_k$. 
The following result is \cite[Theorem 9.1.6]{MR0252485}.
\begin{lemma}\label{Lemma PF} The coradical filtration satisfies 
$$\Delta (F_n) \subset  \sum_{i=0}^n  F_{i}\otimes F_{n-i}\,.$$
\disqed
\end{lemma}

The following two lemmas are well known for irreducible coalgebras.

\begin{lemma} For $g\in G$ define the subspace of $g$-primitives by 
$$P_g=\{ c\in C \mid \Delta(c)=c\otimes g + g \otimes c\}\,.$$  
Then the inclusion map
$$R\oplus \Bigl( \bigoplus_{g\in G} P_g\Bigr) \lra F_1$$
is an isomorphism.
\end{lemma}

\begin{proof}
By the previous lemma we can write the diagonal of  $c\in F_1$ in the form
$$ \Delta (c) = \sum a_g \otimes  g + g \otimes b_g $$
for suitable $a_g,b_g$ in $F_1$. 
Coassociativity and linear independence of grouplike elements yield for 
each $g\in G$ the equality 
$$ \Delta a_g\otimes g +g \otimes g \otimes b_g = a_g\otimes g \otimes g +g \otimes \Delta b_g\,.$$
Applying $\epsilon \otimes 1 \otimes \epsilon$ to this equation we get
$$ c=a_g+\epsilon (b_g) g=\epsilon(a_g)g + b_g\,.$$
Set $\tilde{c}=c-\sum \epsilon(a_g+b_g)g$. Then 
$$\Delta(\tilde{c}) = \sum (a_g-\epsilon(a_g)\otimes g + g \otimes (b_g-\epsilon(b_g)g)=
\sum \tilde{c}\otimes g+g\otimes \tilde{c}$$
and we conclude 
$\tilde{c}$ lies in $\bigoplus_{g\in G} P_g$.
Hence the map is surjective. Injectivity follows from the linear independence of the 
grouplike elements.
\end{proof}

\begin{convention}
In the sequel we suppose that  $\Sigma$ is a pointed coalgebra over $\mk$ with an exhaustive 
coradical filtration. 
\end{convention}

Let $G$ be a group and suppose we are given an injective algebra map
$$\mk [G]\lra \Sigma$$
whose image is the set of  grouplike elements. Then one verifies  that the canonical map
$$ \varphi: \mk[G]\otimes P_1 \lra \bigoplus_{g\in G} P_g,\quad 
g\otimes \sigma \longmapsto g\sigma$$
is an isomorphism.

Next, suppose $M$ is a right $\Sigma$-comodule with coaction $\psi$. 
Define a filtration $F_k$ for $k\geq 0$ on $M$ by
$$F_k(M)=\psi^{-1}(M\otimes F_{k})\,.$$
Note that this filtration is preserved by maps $f:M\ra M'$ of $\Sigma$-comodules. 
Indeed, if $m$ has filtration $k$ in $M$ and if $\psi'$ denotes the diagonal of $M'$  then
$$ \psi' f(m)=(f\otimes \id)\psi(m)\in f(M)\otimes F_{k}\subset M'\otimes F_{k}$$
and hence $f(m)\in F_k(M')$.
\begin{lemma}
For $g\in G$ define the space of $g$-primitives by
$$ P_g( M)=\{ m\in M | \, \psi (m)= m\otimes g\}$$
and let the space of primitives $P(M)$ be generated by all $P_g(M)$. 
Then the maps  
$  \oplus_{g\in G}P_g(M) \ra P(M) $ and 
$P(M) \ra F_0(M)$ induced by the inclusions
are  isomorphisms.
\end{lemma}
\begin{proof}
The first isomorphism follows from the linear independence of the grouplike elements. 
It remains to show surjectivity of the second map. Suppose  $m\in F_0(M)$. 
Then we can write $\psi(m)$ in the form
$$ \psi(m)=\sum_g m_g \otimes g$$
for some $m_g\in M$. When applying $1\otimes \epsilon$ to this equation we obtain
$$m=\sum_g m_g\,.$$
Hence, it suffices to show that $m_g$ lies in $P_g(M)$. Coassociativity implies
$$\sum_g\psi(m_g)\otimes g = \sum_g m_g \otimes g \otimes g$$
and thus $\psi(m_g)=m_g\otimes g$.
\end{proof}
There is another way to think of $g$-primitives. Recall from \cite[A.1.1.4]{MR860042}  that 
the cotensor product $M\Box_\Sigma  M'$ of a right comodule $(M,\psi,\epsilon)$ with a left 
comodule $(M',\psi',\epsilon')$ is defined as the equalizer
of  $\psi {\otimes}  \id_{M'}$ and  $\id_M {\otimes}  \psi'$ .
One readily verifies that the maps
\begin{eqnarray*}
P_g(M) &\lra & M\Box_\Sigma \mk g; \quad m \mapsto m\otimes g \\
P(M)&\lra &  M\Box_\Sigma R; \quad (\sum_g m_g)  \mapsto ( \sum_g  m\otimes g)
\end{eqnarray*}
are isomorphisms.

\begin{lemma} \label{key1}
$\psi(F_n(M))\subset P(M) \otimes F_n +M\otimes F_{n-1}\,.$
\end{lemma}

\begin{proof} 
Choose a set of representatives $\sigma$ for a basis in $F_n/ F_{n-1}$. 
For $m \in F_{n}(M)$ write $\psi(m)$ in the form 
$$ \psi(m) = \sum_\sigma m_\sigma \otimes \sigma  + \mbox{ terms in } M\otimes F_{n-1}\,.$$
Then coassociativity and Lemma \ref{Lemma PF}  yield modulo terms in  $ M\otimes F_{n-1}$
\begin{align*}
\sum_\sigma \psi(m_\sigma) \otimes \sigma &= \sum _\sigma m_\sigma \otimes \psi (\sigma)
= \sum _{\sigma,g} m_\sigma \otimes g\otimes \sigma_g.
\end{align*}
Let $\left< \sigma_g ,  \tau \right> $ be the coefficient of $\sigma_g$ with respect to 
the basis element $\tau$. Then the last equation gives
$$ \psi(m_\sigma)= \sum_{\tau,g}m_\tau \otimes \left< \tau_g, \sigma\right>g\, 
\subset M\otimes F_0\,.$$
Thus $m_\sigma$ has filtration 0, whence the claim. 
\end{proof}

In case  $M=\Sigma$ we have the two filtrations $F_k$ and $F_k(\Sigma)$. 
There is yet another filtration given by 
$$\bar{F}_k(\Sigma)= \psi^{-1}(F_{k}\otimes \Sigma)\,.$$
Fortunately, Lemma \ref{Lemma PF} says that
all three filtrations agree.

\begin{definition}
Suppose the comodule $M$ is equipped with maps of comodules $\eta:\mk \lra M$ and  
$\epsilon\colon M\lra \mk$ which satisfy $\epsilon \eta = \id$. We write $1$ for the image of~1 
under $\eta$ as well. Define the graded left primitives by 
$$\bar{P}_1Gr_{k}(M)=\{ m\in M | \,  \psi (m) 
=1 \otimes \sigma \text{ mod } M\otimes F_{k-1} \text{ for some } \sigma \in F_{k} \}\,.$$
We say a map of comodules $f$ is $\star$-{\em surjective} if it is surjective and the induced 
map on graded left primitives is surjective for all $k$.
\end{definition}

\begin{lemma}\label{lemma2.6}
Suppose $M$ is a  $\Sigma$-comodule $\mk[G]$-algebra. Then the map
$$ \varphi: \mk[G]\otimes P_1(M) \lra P(M); \quad g \otimes m \mapsto gm $$
is an isomorphism. 
\end{lemma}

\begin{proof}
For $m\in P_1(M)$ the calculation
$$\psi(gm)=(g\otimes g)(m\otimes 1 ) = gm \otimes g$$
shows $gm\in P_g(M)$. An inverse map is given by 
$$ \varphi^{-1}(m)=\varphi^{-1}(\sum_gm_g)=\sum_gg \otimes (g^{-1}m_g)\,.$$
\auxqed[1]
\end{proof}
A version of the classical Milnor-Moore theorem is the statement that a  graded connected Hopf algebra is free over each of its sub Hopf algebras. A generalization of the dual statement for comodules is \cite[Proposition 2.6] {MR0174052}. For graded connected Hopf algebroids the result can be found in \cite[A1.1.17]{MR860042}. The following result is a generalization to comodules over pointed coalgebras with an exhaustive coradical filtration.
\begin{thm}\label{MM}
Let $M$ be a  right $\Sigma$-comodule $\mk[G]$-algebra. Suppose that $G$ is grouplike in $M$. 
Let $f\colon M\lra \Sigma$ be a $\Sigma$-comodule and $\mk[G]$-algebra map which is 
$\star $-surjective. Then  there is an isomorphism of $\Sigma$-comodules
$$h: M \lra P_1(M) \otimes \Sigma\,.$$
\end{thm}

\begin{proof}
Choose a linear  left  inverse  $r$   of the inclusion map $i:  P(M)\lra M$. 
Define $h$ as the composite
$$M\stackrel{\psi}{\lra}M\otimes \Sigma \xrightarrow{r \otimes\id} P(M)\otimes \Sigma
\xrightarrow{\varphi^{-1}\otimes\id} \mk[G]\otimes P_1(M) \otimes \Sigma 
\xrightarrow{\epsilon\otimes\id\otimes\id}P_1(M) \otimes \Sigma\,$$
with $\varphi$ as in Lemma \ref{lemma2.6}.
First we show that $h$ is injective. Let $m\in F_{n}(M)$ be in the kernel of $h$.
As in Lemma \ref{key1}  write  $$\psi(m) = \sum m_\sigma \otimes \sigma + \mbox{ terms in } M\otimes F_{n-1}$$
with $m_\sigma \in P(M)$. Write  $m_\sigma=\sum_gm_{\sigma, g}$, $m_{\sigma , g}\in P_g(M)$. Again, coassociativity implies modulo terms of lower filtration the equality 
$$ \sum_\sigma m_{\sigma,g} \otimes \sigma = \sum _\sigma m_{\sigma, g} \otimes \sigma_g\,.$$ Here, $\sigma_g$ is the term which comes up in 
$$ \psi(\sigma)= \sum_g g \otimes \sigma_g + \text{ terms of lower filtration.}$$
Calculating modulo  $ M\otimes F_{n-1}$
\begin{align*}
h(m) &= (\epsilon \otimes \id\otimes \id)(\varphi^{-1}r \otimes \id)\psi(m) 
=  \sum_{\sigma,g} (\epsilon \otimes \id\otimes \id ) \varphi^{-1}(m_{\sigma,g} ))\otimes \sigma_g 
\\ &=  \sum_{\sigma, g} (\epsilon \otimes \id\otimes \id ) 
  (g\otimes g^{-1}m_{\sigma ,g})\otimes \sigma_g
=  \sum_{\sigma, g}  g^{-1}m_{\sigma ,g}\otimes \sigma_g.
\end{align*}
Set
$$Gr_{n,g}=\{ \sigma \in F_n/F_{n-1}\mid \psi(\sigma) 
 = g \otimes \sigma  \mbox{ mod } M\otimes F_{n-1} \} \,.$$
The map
\begin{align*}
\bigoplus_{g\in G} P_g(M) \otimes Gr_{n,g} &\lra  P_1(M) \otimes  \bigoplus_{g\in G} Gr_{n,g}\\
m_g\otimes \sigma_g &\longmapsto g^{-1}m_g \otimes \sigma_g
\end{align*}
is an isomorphism. Thus the calculation of $h(m)$ implies that $\psi(m)$ vanishes modulo terms 
in $M \otimes F_{n-1}$. This means that $m$ has filtration $m$ and an obvious induction shows 
that $h$ is injective.

It remains to show that $h$ is onto. For an element $\sigma\in F_n$ write modulo terms in 
$\Sigma \otimes F_{n-1}$ 
$$\psi(\sigma)=\sum g \otimes \sigma_g $$
with $\sigma =\sum_g \sigma_g$ and $\psi(\sigma_g) = g \otimes \sigma_g$. 
It suffices to show that $n\otimes \sigma_g$ for $n\in P_1(M)$ lies in the image of $h$ 
modulo terms in  $P_1(M) \otimes F_{n-1}$ . Choose an inverse $m$ of $g^{-1}\sigma_g$ 
in $\bar{P}_1Gr_n(M)$. Then modulo those terms we have
$\psi(gm) = g \otimes \sigma_g$ and hence
$$ h(gmn)= (\epsilon\otimes \id) \varphi^{-1} (gn)\otimes \sigma_g = n \otimes \sigma_g\,.$$
The claim now follows by induction.
\end{proof}

\section{The spectrum $T_2\wedge MString$.}
Let $p$ be a prime and let $E=E_n$ be the height $n$ Morava $E$-theory spectrum 
at the prime $p$. 
Let $\mathfrak m=(p,u_1,u_2,\ldots )$ be the unique homogeneous maximal ideal 
in the coefficient ring
$$E_*\cong W\F_{p^n}[\![u_1,u_2,\ldots u_{n-1}]\!][u,u^{-1}]\,.$$
Then $K=E/\mathfrak m$ is a variant of the Morava $K$-theory spectrum with the same
Bousfield localization as $K(n)$. The group $\Gamma$ of ring spectrum automorphisms 
of~$E$ is a version of the Morava stabilizer group.

\begin{thm}\cite{MR2076002}
The inclusion of $\Gamma$ as a subgroup of $E^0E$ induces  an isomorphism of the completed group ring 
$$E^*[\![\Gamma]\!] \lra E^*E\,.$$
Dually, we have the isomorphism $$E_*^\vee E\lra C(\Gamma,E_*)$$
between the homotopy groups of the $K(2)$-localized product $E\wedge E$ 
and  the ring of continuous maps from the profinite group $\Gamma$ to $E_*$.
\end{thm}

The objects $(E_* ,E^\vee_*E)$ and $C(\Gamma,E_*)$ are graded formal Hopf algebroids. 
In fact, they are evenly graded and since $E_*$ has a unit in degree 2 we may as well restrict 
our attention to the degree 0 part. Let $V$ be $BP_*$ as an ungraded ring and let $VT$ be 
the ungraded ring $BP_*BP$. Then Landweber exactness furnishes the equivalence
$$E_0^\vee E\cong E_0\otimes_V VT[t_0,t_0^{-1}] \otimes_V E_0.$$
The algebra  $K_0 E$ is the reduction modulo the maximal ideal $\mathfrak m$ and hence itself carries the structure of a Hopf algebroid. 
(cf. \cite[Proposition 3.8]{MR2076002}). \par
The paper at hand deals with the case $p=2$ and $n=2$. 
Here, it will prove useful to employ a version of the Morava $E$-theory which comes from the 
deformation of the elliptic curve
$$C:\quad y^2+y=x^3\,.$$
The Hopf algebroid  $\Sigma = K_0 E$ is not pointed when we consider 
the curve over $\F_4$. Hence, we will only consider it over $\F_2$ so that 
$$E_*=\Z_2[ \![u_1,u_2,\ldots u_{n-1}]\!][u,u^{-1}]\,.$$
The results stated above hold without changes. Since now left and right unit coincide modulo the ideal $\mathfrak m$ the ring $\Sigma  =  K_0 E$ actually carries the structure of a Hopf algebra.
Explicitly we have 
(see \cite{MR2076002})
$$\Sigma  = \F_2[t_0,t_1,\ldots]/(t_0^3-1,t_2^4-t_1,t_2^4-t_2,\ldots )\,.$$

\begin{lemma}
The  Hopf algebra $\Sigma$ is pointed and has an exhaustive coradical filtration.
\end{lemma}

\begin{proof}
The elements $1,t_0,t_0^2$ are grouplike. 
All classes $t_i$ with $ i\geq 1$ come from the graded Hopf algebroid $BP_*BP$. (More precisely, the generators $t_i$ in $\Sigma$ are obtained  from those of $BP_*BP$ by multiplication with a suitable power of the periodicity element to place them in degree zero.) A complete formula for the diagonal is given in \cite[A2.1.27]{MR0219077}, but from the grading it is already clear that 
\begin{eqnarray}\label{Delta}
\Delta (t_i) &=& 1\otimes t_i +t_i \otimes 1 + \text{ terms involving only $t_1,t_2,\ldots t_{i-1}$.}
\end{eqnarray}
Consider the surjection of coalgebras
$$ \rho\colon\F_2[\Z/3][t_1,t_2,\ldots ] \lra \Sigma$$
which sends the generator 1 of $\Z/3$ to $t_0$. Note that the source is in fact a connected graded coalgebra.
Suppose $S$ is a non trivial simple subcoalgebra of $\Sigma$. Then $\rho^{-1}S$ is a subcoalgebra of the polynomial algebra. From formula \ref{Delta} it is clear that it contains a grouplike element and so does $S$. Since $S$ is simple it has to coincide with the one-dimensional subcoalgebra generated by this element. It follows that $\Sigma$ is pointed.\par
To show exhaustiveness one argues similarly: formula \ref{Delta} implies that for each 
monomial in the $t_i$,  the maximal degree of a tensor factor, modulo elements in the 
group ring, decreases each time the diagonal $\Delta$  is applied. Hence the 
polynomial ring is coradically exhaustive and so is $\Sigma$.
\end{proof}

Recall from \cite[Section 6.5]{MR860042} the ring spectrum $T(2)$ which is part of a 
filtration of $BP$ and which was already mentioned in the introduction. 
Equation \eqref{T2} defines the ring spectrum $T_2$. 
Note that $T_2$ comes with a canonical map
$$\can\colon T_2 \lra (L_{K(2)}BP)[u^\pm]/(v_2-u^3)\lra E\, .$$
Set
\begin{align*}
M&= K_0 (T_2\wedge MString)\\
M^\CC &= K_0 (T_2\wedge MU\langle 6 \rangle).
\end{align*}
Then $M$ and $M^\CC$ are right $\Sigma$-comodule $\F_2[\Z/3]$-algebras. 
Moreover, we have the Witten orientation
$$\tau_W: MString \lra E$$
and the complex orientation induced by the standard coordinate on the elliptic curve $C$
$$ \tau_U: MU\langle 6\rangle \lra MU \lra E$$
(for details see e.g.~\cite{MR3448393}). The composite 
$$\tau^\CC_W : MU\langle 6\rangle\lra MString \stackrel{\tau_W}{\lra} E$$
gives another $MU\langle 6\rangle$-orientation on $E$. The difference class 
$$r_U= \tau^\CC_W /\tau_U \in E^0(BU\langle 6\rangle)$$
plays an important role in the theory of string characteristic classes (loc.\ cit.).
The orientations induce maps

\begin{alignat*}{2}
{\tau_W}_*&: &\quad &M =K_0 (T_2\wedge MString)
\xrightarrow{(\can\wedge\tau_W)_*} K_0 (E_2\wedge E_2)\stackrel{\mu_*}{\lra} K_0(E)=\Sigma\\
{\tau_U}_*&: &\quad &M^\CC =K_0 (T_2\wedge MU\langle 6\rangle)
\xrightarrow{(\can\wedge\tau_W^{\CC})_*} K_0 (E_2\wedge E_2)\stackrel{\mu_*}{\lra} K_0(E)=\Sigma
\end{alignat*}

\begin{lemma}\label{key}
There is a $\Sigma$-comodule algebra automorphism $\alpha$ of $M^\CC$ with the property that
$$\xymatrix{M^\CC \ar[r]^\alpha \ar[d]&M^\CC\ar[d]^{{\tau_U}_*}\\
M\ar[r]^{{\tau_W}_*}&\Sigma}$$
commutes.
\end{lemma}


\begin{proof}
Ando-Hopkins-Strickland have shown in \cite{MR1869850} that a ring map from 
$BU\langle 6\rangle_+$ to a complex oriented ring spectrum coincides with a cubical structure 
on the associated formal group. In particular, such a map is determined by its restriction to
$$ P=\CC P^\infty \wedge \CC P^\infty  \wedge \CC P^\infty\,.$$  
The class $r_U \in E^0BU\langle 6 \rangle$ corresponds to such a cubical structure and 
hence satisfies the cocycle condition when restricted to $P$. We claim that the composite
$$r\colon BU\langle 6\rangle_+ \stackrel{r_U}{\lra } E \stackrel{\eta_R}{\lra} K\wedge E$$
already maps to $K\wedge T_2$. This is clear when it is restricted to $P$ because the 
coefficients of the power series already lie in $\eta_R(\pi_*E)$ and hence in the subring 
$$\eta_R(\pi_*T_2)\subset \pi_*(K\wedge T_2)$$ 
where $\eta_R$ denotes the right unit. Since the restriction map satisfies the cocycle condition 
it comes from a map $\hat{r}$ in $(K\wedge T_2)^0BU\langle 6 \rangle $. 

Consider the commutative diagram 
$$
\xymatrix{K\wedge T_2\wedge MU\langle 6\rangle \ar[d]^{1\wedge \Delta}\ar[r]&
K\wedge E\wedge MU\langle 6\rangle \ar[d]^{1\wedge \Delta}\\
K\wedge T_2\wedge BU\langle 6\rangle _+\wedge MU\langle 6\rangle \ar[d]^{1\wedge \hat{r}\wedge 1}\ar[r]&
K\wedge E\wedge BU\langle 6\rangle _+\wedge MU\langle 6\rangle \ar[d]^{1\wedge {r}\wedge 1}\\
K\wedge T_2\wedge K\wedge T_2\wedge MU\langle 6\rangle \ar[d]^{\mu\wedge 1}\ar[r]&
K\wedge E\wedge K\wedge E\wedge MU\langle 6\rangle \ar[d]^{\mu\wedge 1}\\
K\wedge T_2 \wedge MU\langle 6\rangle\ar[r] &
K\wedge E\wedge MU\langle 6\rangle }
$$
in which $\Delta$ is the Thom diagonal. The automorphism $\alpha$ is induced by the 
composite of the vertical maps on the left.  Since the bottom horizontal map is injective in homotopy 
it suffices to prove commutativity of the claimed square for the map induced by $r$ instead of 
$\hat{r}$. This in turn follows immediately from the equality
$$ \tau_W = r_U \tau^\CC_U\,.$$
It remains to show that $\alpha$ is a $\Sigma$-comodule automorphism which again is 
obvious for the the map induced by $r$ instead of $\hat{r}$.
\end{proof}

\begin{prop}\label{Hopfsplitting}
The maps ${\tau_U}_*$ and ${\tau_W}_*$ are $\star$-surjective. In particular, 
there are isomorphisms of $\Sigma$-comodules
\begin{align*}
h^\CC :\:  &M^\CC \lra P_1(M^\CC) \otimes \Sigma  \\
h: \: &M  \lra P_1(M) \otimes \Sigma 
\end{align*}
\end{prop}
\begin{proof}
As already mentioned in the inroduction, Ravenel and Hovey show in 
\cite[Corollary 2.2(2)]{MR1297530} that at $p=2$ the spectrum  
$DA(1)\wedge MU\langle 6 \rangle$ is a wedge of suspensions of $BP$. A closer inspection 
of the proof reveals that the splitting isomorphism can be chosen to make the diagram
$$\xymatrix { DA(1)\wedge MU\langle 6 \rangle \ar[r]^\simeq \ar[d]& \bigvee_{i\in I} \Sigma^{n_i}BP\ar[d]\\ 
MU \wedge BP\ar[r] & BP}$$
commutative. Here the right vertical map is the projection onto the summand which contains 
the unit in homotopy. Since  $K(2)$-locally each $BP$-summand splits further into sums of 
Johnson-Wilson spectra $E(2)$ it follows that there is a section of the composite
$$ g: DA(1)\wedge MU\langle 6 \rangle \lra BP \lra E(2)\,.$$
Let $j$ be the map from $E(2)$ to $E$. Then  $jg$ factors over $T_2 \wedge MU\langle 6 \rangle$  
and we obtain a section $s$ of the canonical map 
$$ T_2 \wedge MU\langle 6 \rangle \lra E\,.$$
Now it is clear that ${\tau_U}_*$ is $\star$-surjective. Moreover, Lemma \ref{key} implies 
the same is true for the map  ${\tau_W}_*$. Hence, the second claim is a corollary of 
Theorem~\ref{MM}.
\end{proof}

\begin{prop}\label{cofree}
The module $K_*(T_2\wedge MString) $ is concentrated in even dimensions.
\end{prop}
\begin{proof} It is clear that $K_* T_2$ is concentrated in even dimensions. 
For the module $$K_* MString\cong K_* BString,$$ this is the main result of \cite{MR2093483} (see also \cite[Remark 3.1, Theorem 1.3]{MR3471093} for this version of Morava $K$-Theory). The claim follows from the K\"unneth isomorphism.
\end{proof}

\begin{prop}\label{last}
Let $X$ be a $K(n)$-local spectrum whose Morava K-homology is concentrated in even degrees. Then $X$ is a wedge of copies of $E$ if and only if $K_0(X)$ is cofree as a comodule over $\Sigma=K_0(E)$. 
\end{prop}
\begin{proof} Obviously, the cofreeness as a comodule is a necessary condition for a splitting. 
In order to show the converse, we first observe with \cite[Proposition 8.4(e)(f)] {MR1601906} that the Morava $K$-homology of $X$ being even implies 
the profreenes of $E_*^\vee(X)$ and 
$$K_*(X)\cong E_*^\vee(X) /\mathfrak m .$$
Choose a wedge $F$ of copies of $E$ and an isomorphism of $\Sigma$-comodules
$\bar{\alpha}$ from $K_0(X)$ to $K_0(F)$. Also, choose a 
lift of $\bar{\alpha}$  in the diagram 
$$\xymatrix{E_0^\vee (X)\ar@{-->}[r]^\alpha \ar[d]^{\text{mod } \mathfrak m} &E_0^\vee (F)\ar[d]^{\text{mod }\mathfrak m} \\
K_0 (X) \ar[r]^{\bar{\alpha}} & K_0 (F)}\,.$$
The universal coefficient theorem for $E$-module spectra yields an isomorphism 
$$\alpha^*\colon  F^*(F) \cong \Hom_{E_*}(E^\vee_*F,F_*)\stackrel{\cong}{\lra}\Hom_{E_*}(E^\vee_*(X),F_*)\cong F^*(X).$$
Let $f\colon X\ra F$ be the image of 1 under this map. In other words, $f$ corresponds to the
composite
$$ E^\vee_*(X) \stackrel{\alpha_*}{\lra} E_*^\vee (F) \stackrel{\mu_*}{\lra}F_*$$
where $\mu$ is the multiplication.
We claim that $f$ induces the map $\bar{\alpha}$ in $K$-homology and hence is a $K$-local isomorphism. Let $p\colon F\ra E$ be the projection onto a summand. It suffices to show the equality
$$ p_*f_*=p_*\bar{\alpha}: K_0X\lra K_0E$$
of $\Sigma$-comodule maps or, dually, the equality of $K^0(E)$-module maps from $K^0E$ to $K^0X$. By construction, the maps $p_*f_*$ and $p_*\bar{\alpha}$ coincide when composed with the augmentation $\mu_*$ to $K_*$. Hence, the dual maps coincide 
on the generator $1\in K^0(E)$ and the result follows.
\end{proof}

\begin{proof}[ of Theorem \ref{main}] 
The theorem is a consequence of the Propositions \ref{Hopfsplitting}, \ref{cofree} and \ref{last}.
\end{proof}

\bibliographystyle{amsalpha}
\bibliography{toda}
\end{document}